\documentclass[11pt]{article}
\usepackage[T1]{fontenc}
\usepackage[utf8]{inputenc}
\usepackage{lmodern}
\usepackage{amsmath,amssymb,amsthm}
\usepackage{microtype}
\usepackage{geometry}
\usepackage{enumitem}
\newtheorem{theorem}{Theorem}
\newtheorem{proposition}{Proposition}
\newtheorem{lemma}{Lemma}
\newtheorem{conjecture}{Conjecture}

\newcommand{\cS}{\mathcal{S}}
\newcommand{\cA}{A}
\newcommand{\eps}{\varepsilon}

\title{On the Maximum Size of 2-Weakly Compatible Split Systems}
\author{Pei Wu \and Stefan Gr\"unewald}
\date{}

\begin{document}
	\maketitle
	
	\begin{abstract}
		We consider a Tur\'an-type problem arising in phylogenetics: determining the maximum size of a 2-weakly compatible split system. This compatibility condition arises in the reconstruction of phylogenetic networks from quartet weights. It was previously shown that a 2-weakly compatible split system has size at most
		\[
		3\binom{n}{4}+\binom{n}{2}.
		\]
		We prove that the maximum size is $O(n^{5/2})$.
	\end{abstract}
	
	\section{Introduction}
	
	In extremal graph theory, a classical result of Tur\'an~\cite{Turan1941} determines the maximum number of edges in a graph on $n$ vertices that contains no complete subgraph $K_r$ for $r>2$: the extremal number is
	\[
	\left(\frac{r-2}{r-1}+o(1)\right)\binom{n}{2}.
	\]
	More generally, for any fixed forbidden graph $H$, one asks for the maximum number of edges in an $n$-vertex graph containing no copy of $H$; this is the Tur\'an number $\operatorname{ex}(n,H)$. This framework extends naturally to set systems: given a family of forbidden configurations, what is the maximum size of a structure that avoids them? This quantity is denoted by $\operatorname{forb}(m,F)$, where $F$ is either a single forbidden configuration or a family of such configurations. For a survey of this type of problem, see~\cite{AnsteeSali2013}. 
	
	The VC-dimension bound/Sauer--Shelah theorem~\cite{Sauer1972,Shelah1972,VapnikChervonenkis1971} implies that $\operatorname{forb}(m,F)$ is polynomially bounded in $m$ for every nonempty $F$, even the number of all possible subset is $2^m$. 
	A central conjecture of Anstee and Sali predicts asymptotically optimal constructions, from which the asymptotic growth of $\operatorname{forb}(m,F)$ for a single forbidden configuration $F$ can be derived.

	We consider split systems, the projective analogue of set systems, in which the elements are bipartitions of the ambient set rather than subsets. We study a Tur\'an-type problem with forbidden configuration
	\[
	\{123|456,124|356,125|346,126|345\}.
	\]
	A split system that avoids this configuration is called \emph{2-weakly compatible}. This notion arises in phylogenetics: 2-weakly compatible split systems are precisely those that can be reconstructed consistently from quartet data~\cite{YangGrunewaldWan2013}.
	
	A linear-algebra argument~\cite{YangGrunewaldWan2013} shows that the size of a 2-weakly compatible split system is at most
	\[
	3\binom{n}{4}+\binom{n}{2}.
	\]
	On the other hand, circular split systems yield a lower bound of $\binom{n}{2}$. Our main result substantially improves the upper bound.
	
	\begin{theorem}\label{thm:main}
		If $\#X=n$, then the maximum size of a 2-weakly compatible split system on $X$ is $O(n^{5/2})$.
	\end{theorem}
	
	\subsection{Background from phylogenetics}
	
	The usual goal of a phylogenetic analysis is to reconstruct a leaf-labelled tree from data, most commonly DNA or other biomolecular sequences. The leaves correspond to taxonomic units (taxa), whereas the internal vertices correspond to speciation events. Split systems arise naturally from a phylogenetic tree: each edge $e$ defines a bipartition of the taxon set, namely the two subsets of taxa corresponding to the components obtained by deleting $e$. In phylogenetics, bipartitions are usually called \emph{splits}. A collection of splits of a taxon set $X$ is called \emph{compatible} if all its splits can be obtained from a single phylogenetic tree. A classical result of Buneman~\cite{Buneman1971} states that a split system is compatible if and only if, for every pair of splits, one can choose one part from each split such that the two chosen parts have empty intersection. Equivalently, a compatible split system avoids the forbidden configuration $\{12|34,13|24\}$. The maximum cardinality of a compatible split system on $n\ge 2$ taxa is $2n-3$, attained by a tree in which every internal vertex has degree~3.
	
	The split-system formulation allows trees to be extended to networks that can accommodate reticulate evolution or ambiguity in the data. In this setting, the compatibility condition can be relaxed, permitting split systems more general than compatible ones. The first method based on this approach was SplitDecomposition, introduced by Bandelt and Dress~\cite{BandeltDress1992}. It permits the consistent reconstruction of split systems with the following property: a split system $\cS$ is \emph{weakly compatible} if and only if it avoids the forbidden configuration $\{12|34,13|24,14|23\}$. The maximum cardinality of a weakly compatible split system on $n\ge 2$ taxa is $\binom{n}{2}$, attained by circular split systems. Since SplitDecomposition uses pairwise distances as input, this bound is best possible.
	
	The more recent Quartet-Net method~\cite{YangGrunewaldWan2013} requires more information than pairwise distances and can reconstruct split systems that are more general than weakly compatible systems. It requires the precomputation of
	\[
	3\binom{n}{4}+\binom{n}{2}
	\]
	values from the raw data and consistently reconstructs split systems with the following property: a split system $\cS$ is 2-weakly compatible if and only if it avoids the forbidden configuration $\{123|456,124|356,125|346,126|345\}$.
	
	\subsection{Notation and terminology}
	
	Throughout the paper, $A+x$ denotes $A\cup\{x\}$, and $A-x$ denotes $A\setminus\{x\}$. We occasionally omit the outer braces of sets when no confusion can arise.
	
	\section{Preliminaries}
	
	Let $X$ be a finite ambient set. A \emph{split} is a bipartition $\{A,B\}$ of $X$. For $A=\{a_1,\ldots,a_k\}$ and $B=\{b_1,\ldots,b_l\}$, we write the split as $A|B$ or $a_1\cdots a_k|b_1\cdots b_l$; throughout, we identify $A|B$ with $B|A$. The \emph{size} of $A|B$ is $\min\{k,l\}$, and a split of size 0 or 1 is called \emph{empty} or \emph{trivial}, respectively. Although the empty split is usually not regarded as a split, we include it for convenience. A collection of splits of $X$ is called a \emph{split system on $X$}.
	
	For a split system $\cS$ on $X$ and $X'\subseteq X$, the restriction of $\cS$ to $X'$ is denoted by
	\[
	\cS|_{X'}:=\{(A\cap X')|(B\cap X'):A|B\in\cS\}.
	\]
	Two split systems $\cS_X$ on $X$ and $\cS_Y$ on $Y$ are \emph{isomorphic} if one can be obtained from the other by permuting the labels. A class of split systems, that is, a set of finite split systems, is called \emph{closed} if, for every split system $\cS$ in the class, every subsystem $\cS'\subseteq\cS$ and every restriction of $\cS$ also belong to the class.
	
	Given a split system $\mathcal{F}$ on $\{1,\ldots,k\}$ for some positive integer $k$, the class of all split systems for which no restriction to a $k$-element set contains a split system isomorphic to $\mathcal{F}$ is clearly closed. In this case, $\mathcal{F}$ is called the \emph{forbidden configuration} defining the class. In this notation, the forbidden configuration for compatible split systems is $\{12|34,13|24\}$, whereas that for weakly compatible split systems is $\{12|34,13|24,14|23\}$. For every positive integer $k$, we define the class of \emph{$k$-weakly compatible split systems} by the forbidden configuration
	\[
	\left\{\{1,\ldots,k\}\cup\{i\}\,\middle|\,\{k+1,\ldots,2k+2\}\setminus\{i\}: k+1\le i\le 2k+2\right\}.
	\]
	Thus, weak compatibility is equivalent to 1-weak compatibility.
	
	A powerful tool for bounding the maximum size of split systems with a given forbidden configuration is to bound the number of $x$-pairs~\cite{DressKlucznikKoolenMoulton2001,DressKoolenMoulton2005,GrunewaldKoolenMoultonWu2012}. For a split system $\cS$ on $X$ and $x\in X$, an \emph{$x$-pair} of $\cS$ is a split $A|B$ of $X-x$ such that $(A+x)|B\in\cS$ and $A|(B+x)\in\cS$. Let
	\[
	\partial_x\cS:=\{A|B:(A+x)|B,\ A|(B+x)\in\cS\}
	\]
	be the split system on $X-x$ consisting precisely of all $x$-pairs of $\cS$. Every split $A|B$ of $X-x$ is the restriction of at most two splits in $\cS$, and equality implies that $A|B$ is an $x$-pair of $\cS$. Hence, the following proposition holds.
	
	\begin{proposition}\label{prop:x-pairs}
		If $\cS$ is a split system on $X$ and $x\in X$, then
		\begin{equation}\label{eq:x-pair-count}
			\#\cS-\#(\cS|_{X-x})=\#\partial_x\cS.
		\end{equation}
	\end{proposition}
	
	Proposition~\ref{prop:x-pairs} holds only because $\cS$ and $\cS|_{X-x}$ are allowed to contain the empty split. Otherwise, the trivial split $x|X-x$ would need to be considered separately. Proposition~\ref{prop:x-pairs} implies that, for $c>-1$, an $O(n^c)$ upper bound on
	\[
	\min\{\#\partial_x\cS:x\in X\}
	\]
	for every split system in a given class on a set $X$ with $\#X=n$ yields an $O(n^{c+1})$ upper bound on $\#\cS$. This approach gives a cubic upper bound for the cardinality of 2-weakly compatible split systems, but we use a new method to obtain the bound $O(n^{5/2})$.
	
	\section{Digraphs representing a split system}
	
	Given a split system $\cS$ on $X$, we define a directed graph that represents $\cS$. Choose a root $r\in X$, and, for each split $s\in\cS$, let $a_r(s)$ be the part of $s$ that does not contain $r$. Set
	\[
	\cA_r(\cS):=\{a_r(s):s\in\cS\}.
	\]
	We define $G_r(\cS)=(V,E)$ to be the edge-labelled digraph with $V=\cA_r(\cS)$ and a directed edge $(a_r(s_1),a_r(s_2))\in E$ for $s_1,s_2\in\cS$ if and only if there exists $x\in X$ such that $a_r(s_2)=a_r(s_1)+x$. Such an edge is labelled $x$. For $x\in X-r$, the $x$-pairs of $\cS$ correspond to the $x$-labelled edges of $G_r(\cS)$, whereas the $r$-pairs are ignored. A subset of the vertices of a graph representing a split system will be called 2-weakly compatible if and only if the corresponding set of splits is 2-weakly compatible.
	
	The main result of this section is the following asymptotic bound on the number of edges in a digraph representing a 2-weakly compatible split system.
	
	\begin{theorem}\label{thm:edge-bound}
		There exists a constant $c$ such that, for every digraph $G=(V,E)$ representing a 2-weakly compatible split system $\cS$ on a set $X$ with $\#X=n\ge 6$,
		\begin{equation}\label{eq:edge-bound}
			\#E\le \frac{5}{2}\#V+cn^2.
		\end{equation}
	\end{theorem}
	
	Before proving the theorem, we establish several properties of the digraph $G$. In the next four lemmas, assume that $G=(V,E)$ represents a 2-weakly compatible split system $\cS$ on a set $X$ with $\#X=n$.
	
	\begin{lemma}\label{lem:degree-bounds}
		Let $A\in V$.
		\begin{enumerate}[label=(\roman*)]
			\item If $\#A>1$, then $d^+(A)\le 3$, and if $\#A<n-1$, then $d^-(A)\le 3$. In particular, $d(A)\le 6$ if $A$ does not correspond to the empty split or a trivial split.
			\item If $A$ corresponds to a 2-split, then $d(A)\le 5$.
			\item If $\#A\in\{0,1,n-1\}$, then $d(A)\le n-1$.
		\end{enumerate}
	\end{lemma}
	
	\begin{proof}
		For (i), suppose that $\#A>1$ and $d^+(A)\ge 4$. Let $A+b_1,A+b_2,A+b_3,A+b_4$ be four distinct out-neighbours of $A$, and let $a_1,a_2$ be two distinct elements of $A$. The restrictions of the four splits corresponding to these out-neighbours to $\{a_1,a_2,b_1,b_2,b_3,b_4\}$ form a system isomorphic to the forbidden configuration for 2-weak compatibility. Similarly, if $d^-(A)\ge 4$ and $X-A$ contains two distinct elements, then the restrictions of the splits corresponding to the in-neighbours of $A$ form a system isomorphic to the forbidden configuration. Combining these two bounds gives the final assertion in (i).
		
		Parts (ii) and (iii) follow from (i) and the immediate bounds $d^-(A)\le\#A$ and $d^+(A)\le n-1-\#A$.
	\end{proof}
	
	A \emph{6-vertex} is a vertex of degree 6. A 6-vertex $A$ is called \emph{hard} if there exist $a_i,b_i\in X$ for $i\in\{1,2,3\}$ such that $A-a_i$ and $A+b_i$ are the six neighbours of $A$, and $A-a_i+b_i\in V$ for each $i$. The vertex $A-a_i+b_i$ is called a \emph{partner} of $A$ if none of the sets $A-a_i+b_j$ and $A-a_j+b_i$, with $j\in\{1,2,3\}\setminus\{i\}$, belongs to $V$.
	
	\begin{lemma}\label{lem:hard-partners}
		For every graph $G$ representing a 2-weakly compatible split system, the following statements hold.
		\begin{enumerate}[label=(\roman*)]
			\item Every hard 6-vertex has at least one partner.
			\item If $A-a+b$ is a partner of a hard 6-vertex $A$, then $d^+(A-a)=d^-(A+b)=2$.
			\item No vertex is a partner of two distinct hard 6-vertices.
		\end{enumerate}
	\end{lemma}
	
	\begin{proof}
		For (i), suppose to the contrary that a hard 6-vertex $A$ has no partner. Let $A-a_i$ and $A+b_i$, for $i=1,2,3$, be its six neighbours, with $A-a_i+b_i\in V$. Without loss of generality, assume that $A-a_1+b_2\in V$. Since $A-a_3+b_3$ is not a partner of $A$, there exists $j\in\{1,2\}$ such that $A-a_3+b_j\in V$ or $A-a_j+b_3\in V$. By symmetry, we may assume that $j=1$. If $A-a_3+b_1\in V$, then the four vertices
		\[
		A-a_3+b_3,\quad A-a_3+b_1,\quad A-a_1+b_2,\quad A-a_2+b_2
		\]
		contradict 2-weak compatibility. If $A-a_1+b_3\in V$, then the vertices
		\[
		A,\quad A-a_1+b_1,\quad A-a_1+b_2,\quad A-a_1+b_3
		\]
		do so.
		
		For (ii), suppose that $A-a+b'\in V$ for some $b'\notin\{a,b\}$. Since $A-a+b$ is a partner of $A$, the set $A+b'$ is not a vertex. The three out-neighbours of $A$, together with $A-a+b'$, are therefore not 2-weakly compatible, contradicting the hypothesis. Hence $d^+(A-a)=2$. The equality $d^-(A+b)=2$ follows by symmetry.
		
		For (iii), suppose that $A-a_1+b_1$ and $A-a_2+b_2$ are two distinct hard 6-vertices and that $A$ is a partner of both. Let $c_1,c_2,d_1,d_2\in X$ be such that
		\[
		A+b_1,\quad A-a_1+b_1+c_1,\quad A-a_1+b_1+c_2
		\]
		are the out-neighbours of $A-a_1+b_1$, and
		\[
		A-a_2,\quad A-a_2+b_2-d_1,\quad A-a_2+b_2-d_2
		\]
		are the in-neighbours of $A-a_2+b_2$. By (ii), $a_1\ne a_2$ and $b_1\ne b_2$. If $b_2\notin\{c_1,c_2\}$ and $a_1\notin\{d_1,d_2\}$, then the vertices
		\[
		A-a_1+b_1+c_1,\quad A-a_1+b_1+c_2,\quad A-a_2+b_2-d_1,\quad A-a_2+b_2-d_2
		\]
		are not 2-weakly compatible. Thus, by symmetry, we may assume that $c_2=b_2$. We must then also have $a_1\in\{d_1,d_2\}$; otherwise, the vertices
		\[
		A,\quad A-a_1+b_1+b_2,\quad A-a_2+b_2-d_1,\quad A-a_2+b_2-d_2
		\]
		are not 2-weakly compatible. Applying the same argument to the in-neighbours of $A-a_1+b_1$ and the out-neighbours of $A-a_2+b_2$, we conclude that $A-a_2+b_1+b_2$ is also a vertex of $G$. Finally, the vertices
		\[
		A+b_1,\quad A+b_2,\quad A-a_1+b_1+b_2,\quad A-a_2+b_1+b_2
		\]
		are not 2-weakly compatible, a contradiction.
	\end{proof}
	
	We now apply the discharging method to show that the average degree of the vertices $A$ satisfying $5\le\#A\le n-5$ is at most 5. Define $C_0(A)=5-d(A)$ for every vertex $A\in V$. We obtain $C_1$ from $C_0$ by applying the following discharging rule: if $d^+(A)=1$ and the unique out-neighbour $A+b$ of $A$ is a 6-vertex or a partner of a hard 6-vertex (or both), then $A$ sends one unit of charge to $A+b$. Similarly, if $d^-(A)=1$ and the unique in-neighbour $A-a$ of $A$ is a 6-vertex or a partner of a hard 6-vertex, then $A$ sends one unit of charge to $A-a$.
	
	\begin{lemma}\label{lem:first-charge}
		Let $A$ be a vertex of a digraph $G=(V,E)$ representing a 2-weakly compatible split system.
		\begin{enumerate}[label=(\roman*)]
			\item If $C_1(A)<0$, then $A$ is a hard 6-vertex.
			\item If $A-a+b$ is a partner of a hard 6-vertex $A$, then $C_1(A-a+b)\ge C_0(A-a+b)$.
		\end{enumerate}
	\end{lemma}
	
	\begin{proof}
		For (i), every vertex that sends one unit of charge to another vertex has degree at most 4, and every vertex that sends one unit of charge to each of two distinct vertices has degree 2. Therefore, $C_1(A)<0$ implies $C_0(A)<0$, so $A$ is a 6-vertex. Suppose that $A$ is a 6-vertex that is not hard, and let $A+b_i$ and $A-a_i$, for $i=1,2,3$, be its neighbours. Every out-neighbour of $A-a_i$ other than $A$ must be of the form $A-a_i+b_j$; otherwise, that out-neighbour together with $A+b_1,A+b_2,A+b_3$ would violate 2-weak compatibility. Similarly, every in-neighbour of $A+b_i$ must be of the form $A-a_j+b_i$.
		
		Consider the bipartite graph $G_A$ with parts $\{a_1,a_2,a_3\}$ and $\{b_1,b_2,b_3\}$, where $a_ib_j$ is an edge precisely when $A-a_i+b_j$ is a vertex of $G$. If $C_1(A)<0$, then $G_A$ has no isolated vertex, because otherwise a neighbour of $A$ would send one unit of charge to $A$. Moreover, $G_A$ cannot contain two vertex-disjoint paths of length 2. For example, if $G_A$ contains the edges $a_1b_1,a_1b_2,a_2b_3,a_3b_3$, then the vertices
		\[
		A-a_1+b_1,\quad A-a_1+b_2,\quad A-a_2+b_3,\quad A-a_3+b_3
		\]
		are not 2-weakly compatible. The graph $G_A$ also cannot contain a vertex of degree 3, because the edges $a_1b_1,a_1b_2,a_1b_3$ would imply that the vertices
		\[
		A,\quad A-a_1+b_1,\quad A-a_1+b_2,\quad A-a_1+b_3
		\]
		are not 2-weakly compatible. Every subgraph of $K_{3,3}$ with no isolated vertices, no vertex of degree 3, and no two vertex-disjoint paths of length 2 has a 1-factor. Hence $G_A$ has a 1-factor, and therefore $A$ is a hard 6-vertex.
		
		For (ii), suppose that $A-a+b$ is a partner of a hard 6-vertex $A$ and that $C_1(A-a+b)<C_0(A-a+b)$. Then $A-a+b$ has in-degree or out-degree 1; by symmetry, assume that $d^+(A-a+b)=1$. Moreover, $A+b$ must be a 6-vertex or a partner of a hard 6-vertex. By Lemma~\ref{lem:hard-partners}(ii), $d^-(A+b)=2$, so $A+b$ cannot be a 6-vertex. It can therefore only be a partner of a hard 6-vertex of the form $A+b'$ for some $b'\in X-b$. In that case, Lemma~\ref{lem:hard-partners}(ii) implies that $d^-(A)=2$, contradicting the fact that $A$ is a hard 6-vertex.
	\end{proof}
	
	We obtain $C_2$ from $C_1$ by applying the following discharging rule: if $A$ is a hard 6-vertex and $A'$ is a partner of $A$ with $C_1(A')\ge 1$, then $A'$ sends one unit of charge to $A$.
	
	\begin{lemma}\label{lem:nonnegative-charge}
		Let $G=(V,E)$ be a digraph representing a 2-weakly compatible split system $\cS$ on $X$, and let $r\in X$ be the root of $G$. Then $C_2(A)\ge 0$ for every vertex $A$ with $5\le\#A\le n-5$.
	\end{lemma}
	
	\begin{proof}
		Suppose first that $C_1(A)\ge 0$. If $A$ sends no charge in the second discharging step, then $C_2(A)\ge C_1(A)\ge 0$. If $A$ sends charge, then $C_1(A)\ge 1$; moreover, Lemma~\ref{lem:hard-partners}(iii) implies that $A$ sends at most one unit. Thus $C_2(A)\ge 0$ in either case. It remains to show that every hard 6-vertex $A$ has a partner $A'$ with $C_1(A')\ge 1$.
		
		Let $A$ be a hard 6-vertex such that $A-a_i$ and $A+b_i$, for $i=1,2,3$, are its six neighbours and $A-a_i+b_i\in V$. Let $A-a_3+b_3$ be a partner of $A$. If $A$ has no other partner, then, without loss of generality, $A-a_1+b_2\in V$. Suppose that $d^+(A-a_3+b_3)=3$, and let
		\[
		A+b_3,\quad A-a_3+b_3+c_1,\quad A-a_3+b_3+c_2
		\]
		be the out-neighbours of $A-a_3+b_3$. If $b_2\notin\{c_1,c_2\}$, then the vertices
		\[
		A-a_3+b_3+c_1,\quad A-a_3+b_3+c_2,\quad A-a_2+b_2,\quad A-a_1+b_2
		\]
		contradict 2-weak compatibility. If $b_2\in\{c_1,c_2\}$, say $b_2=c_2$, then the vertices
		\[
		A-a_2+b_2,\quad A-a_1+b_2,\quad A,\quad A-a_3+b_3+b_2
		\]
		contradict 2-weak compatibility, because the taxa $a_1,a_2,a_3,b_2,b_1,r$ induce the forbidden configuration. Hence $d^+(A-a_3+b_3)\le 2$, and, by symmetry, $d^-(A-a_3+b_3)\le 2$. Thus $C_1(A-a_3+b_3)\ge 1$, and the second discharging rule gives $C_2(A)=0$.
		
		It remains to consider the case in which no vertex $A-a_i+b_j$ with $i\ne j$ and $\{i,j\}\subseteq\{1,2,3\}$ belongs to $V$. Then all three vertices $A-a_i+b_i$, for $i=1,2,3$, are partners of $A$. The assertion holds if one of these vertices has degree at most 4. We may therefore assume that every partner of $A$ has degree at least 5. It follows that two partners of $A$ have out-degree 3 or two have in-degree 3. Without loss of generality, assume that
		\[
		d^+(A-a_1+b_1)=d^+(A-a_2+b_2)=3.
		\]
		For $i=1,2$, let
		\[
		A-a_i+b_i+c_i^1,\quad A-a_i+b_i+c_i^2
		\]
		be the two out-neighbours of $A-a_i+b_i$ other than $A+b_i$. We have
		\[
		b_3\notin \{c_1^1,c_1^2\}\cap\{c_2^1,c_2^2\},
		\]
		because otherwise the vertices
		\[
		A,\quad A+b_3-a_3,\quad A+b_3+b_2-a_2,\quad A+b_3+b_1-a_1
		\]
		would not be 2-weakly compatible: the taxa $a_1,a_2,a_3,b_3$, together with two elements of $X-A-b_1-b_2-b_3$, would induce the forbidden configuration. Hence, without loss of generality, $b_3\notin\{c_1^1,c_1^2\}$.
		
		We have $A+b_3-a_3-a_1\notin V$, because this vertex together with
		\[
		A,\quad A+b_1-a_1+c_1^1,\quad A+b_1-a_1+c_1^2
		\]
		would violate 2-weak compatibility. Moreover, $d^-(A-a_3+b_3)=3$ is impossible; otherwise, the two in-neighbours of $A-a_3+b_3$ other than $A-a_3$, together with $A+b_1-a_1+c_1^1$ and $A+b_1-a_1+c_1^2$, would not be 2-weakly compatible.
		
		We have shown that
		\[
		d^+(A-a_1+b_1)=d^+(A-a_2+b_2)=3
		\]
		implies $d^-(A-a_3+b_3)=2$. By symmetry, and because $d(A-a_i+b_i)\ge 5$ for $i=1,2,3$, we may assume that
		\[
		d^+(A-a_i+b_i)=3\quad\text{and}\quad d^-(A-a_i+b_i)=2
		\]
		for $i=1,2,3$. Using the same argument that allowed us to assume $b_3\notin\{c_1^1,c_1^2\}$, we may also assume that $b_1\notin\{c_2^1,c_2^2\}$. Let $A-a_1+b_1-d_1$ and $A-a_1$ be the in-neighbours of $A-a_1+b_1$. We cannot have
		\[
		\{A-a_1+b_1-d_1+c_1^1,\ A-a_1+b_1-d_1+c_1^2\}\subseteq V,
		\]
		because these two vertices, together with either $A-a_2,A-a_3$ if $d_1\notin\{a_2,a_3\}$, or with $A-a_1,A-d_1$ otherwise, would not be 2-weakly compatible. Assume, without loss of generality, that
		\[
		A-a_1+b_1-d_1+c_1^2\notin V.
		\]
		If $d^-(A-a_1+b_1+c_1^2)>1$, choose $f\in X\setminus\{c_1^2,d_1\}$ such that
		\[
		A-a_1+b_1+c_1^2-f\in V.
		\]
		Then $f\ne b_1$, because otherwise the vertices $A-a_1+c_1^2,A+b_1,A+b_2,A+b_3$ are not 2-weakly compatible if $c_1^2\notin\{b_2,b_3\}$, while $A-a_1+b_1$ is not a partner of $A$ otherwise. Furthermore, $a_2\notin\{d_1,f\}$, because otherwise the vertices
		\[
		A-a_2+b_2+c_2^1,\quad A-a_2+b_2+c_2^2,\quad A+b_2,
		\]
		together with $A-a_1-a_2+b_1$ if $d_1=a_2$, or with $A-a_1-a_2+c_1^2$ if $f=a_2$, would not be 2-weakly compatible. However, $f\ne b_1$, $a_2\notin\{d_1,f\}$, and $b_1\notin\{c_2^1,c_2^2\}$ imply that the vertices
		\[
		A-a_1+b_1+c_1^2-f,\quad A-a_1+b_1-d_1,\quad A-a_2+b_2+c_2^1,\quad A-a_2+b_2+c_2^2
		\]
		are not 2-weakly compatible, because the taxa $a_2,b_1,c_2^1,c_2^2,d_1,f$ induce the forbidden configuration. Therefore,
		\[
		d^-(A-a_1+b_1+c_1^2)=1.
		\]
		Thus $A-a_1+b_1+c_1^2$ sends one unit of charge to $A-a_1+b_1$ in the first discharging step, and $A-a_1+b_1$ sends one unit of charge to $A$ in the second step. Hence $C_2(A)\ge 0$.
	\end{proof}
	
	To prove Theorem~\ref{thm:edge-bound}, it suffices to show that
	\[
	\sum_{A\in V}d(A)\le 5\#V+cn^2
	\]
	for some constant $c$. Equivalently,
	\[
	-\sum_{A\in V}C_0(A)\le cn^2.
	\]
	Both discharging rules preserve the total charge, and Lemma~\ref{lem:nonnegative-charge} gives $C_2(A)\ge 0$ for every vertex satisfying $5\le\#A\le n-5$. Therefore, it remains only to prove
	\[
	-\sum_{\substack{A\in V\\ \#A\ge n-4\text{ or }\#A\le 4}}C_2(A)\le cn^2,
	\]
	where the sum ranges over vertices corresponding to splits of size at most 4.
	
	\begin{proof}[Proof of Theorem~\ref{thm:edge-bound}]
		In the second discharging step, a vertex $A$ can send charge to a vertex $A'$ only if $\#A=\#A'$. Thus, the total charge among vertices of each fixed size is preserved, and
		\[
		-\sum_{\substack{A\in V\\ \#A\ge n-4\text{ or }\#A\le 4}}C_2(A)
		=
		-\sum_{\substack{A\in V\\ \#A\ge n-4\text{ or }\#A\le 4}}C_1(A).
		\]
		By the definition of $C_1$,
		\[
		C_1(A)\ge \min\{0,5-d(A)\}.
		\]
		By Lemma~\ref{lem:degree-bounds}(iii), vertices representing the empty split or a trivial split have degree at most $n-1$. There are at most $n+1$ such vertices, so the sum of $C_1(A)$ over all $A$ with $\#A\in\{0,1,n-1\}$ is at least $-(n+1)(n-6)$. Lemma~\ref{lem:degree-bounds}(i) and (ii) imply that $C_1(A)\ge 0$ when $\#A\in\{2,n-2\}$ and that $C_1(A)\ge -1$ when $\#A\in\{3,4,n-4,n-3\}$. It therefore suffices to bound the numbers of 3-splits and 4-splits.
		
		Let $\cA_3$ and $\cA_4$ be the collections of all 3-sets and 4-sets $A$, respectively, such that $A\in V$ or $X-A\in V$. For every pair of taxa $x_1,x_2\in X$, at most three sets in $\cA_3$ contain $\{x_1,x_2\}$, because $G$ represents a 2-weakly compatible split system. Since every 3-set contains three pairs,
		\[
		\#\cA_3\le\binom{n}{2}.
		\]
		
		To bound $\#\cA_4$, define, for every pair of taxa $x_1,x_2\in X$, a graph $G_{x_1,x_2}$ with vertex set $X-x_1-x_2$. It contains an edge $xx'$ if and only if $\{x_1,x_2,x,x'\}\in\cA_4$. By 2-weak compatibility, there do not exist four edges $e_i$ and four taxa $y_i$, for $i=1,\ldots,4$, such that $y_i$ is incident with $e_j$ if and only if $i=j$. Consequently, the maximum degree of $G_{x_1,x_2}$ is at most 3.
		
		We claim that $G_{x_1,x_2}$ has at most seven edges. Suppose otherwise, and let $G'_{x_1,x_2}$ be an edge-induced subgraph with eight edges. Since eight is not divisible by 3, the graph $G'_{x_1,x_2}$ is not 3-regular. If it contains a vertex $y$ of degree 2, let $y_1,y_2$ be the neighbours of $y$, and set $e_1=yy_1$ and $e_2=yy_2$. Since the degrees of $y_1$ and $y_2$ are at most 3, there are at least two edges $e_3,e_4$ not incident with any of $y,y_1$, or $y_2$. We can then choose vertices $y_3,y_4$ such that, for $i,j\in\{1,\ldots,4\}$, the vertex $y_i$ is incident with $e_j$ if and only if $i=j$, a contradiction.
		
		If $G'_{x_1,x_2}$ contains an edge $yy_1$ such that $y_1$ has degree 1 and $y$ has degree 3, let $y_2,y_3$ be the other neighbours of $y$, and set $e_2=yy_2$ and $e_3=yy_3$. There is then an edge $e_4$ not incident with $y$ or any of $y_1,y_2,y_3$. Choosing $y_4$ to be a vertex incident with $e_4$ again yields a contradiction. The final case, in which $G'_{x_1,x_2}$ contains an isolated edge, yields the same contradiction. This proves the claim.
		
		Thus, for every pair of taxa $x_1,x_2$, at most seven sets in $\cA_4$ contain $\{x_1,x_2\}$. Since every 4-set contains six pairs,
		\[
		\#\cA_4\le\frac{7}{6}\binom{n}{2}.
		\]
		In summary,
		\[
		\sum_{A\in V}C_1(A)
		\ge -(n+1)(n-6)-\frac{13}{6}\binom{n}{2}
		> -\frac{25}{12}n^2.
		\]
		Therefore,
		\[
		\#E\le \frac{5}{2}\#V+cn^2
		\]
		with $c=25/24$.
	\end{proof}
	
	\section{Proof of Theorem~\ref{thm:main}}
	
	We first establish a general result relating the average degree of graphs representing split systems to the sizes of those split systems.
	
	\begin{theorem}\label{thm:recurrence}
		Let $\cS$ be a split system on $X$ with $\#X=n$. Suppose that there exist positive constants $c,k,\ell$ with $\ell<k$ such that, for every $Y\subseteq X$, every graph $G=(V,E)$ representing $\cS|_Y$ satisfies
		\begin{equation}\label{eq:general-edge-bound}
			\#E\le k\#V+c(\#Y)^{\ell}.
		\end{equation}
		Then $\#\cS=O(n^k)$.
	\end{theorem}
	
	After increasing the constant $c$, if necessary, to account for restrictions to sets with fewer than six elements, Theorem~\ref{thm:main} follows from Theorems~\ref{thm:edge-bound} and~\ref{thm:recurrence}.
	
	\begin{lemma}\label{lem:log-concavity}
		If $n>k+1>1$, then
		\[
		\frac{n-1}{n-1-k}\le\left(\frac{n-k}{n-1-k}\right)^k.
		\]
	\end{lemma}
	
	\begin{proof}
		After taking logarithms, it suffices to prove
		\[
		\log(n-1)+(k-1)\log(n-1-k)\le k\log(n-k),
		\]
		which follows from the concavity of the logarithm.
	\end{proof}
	
	\begin{proof}[Proof of Theorem~\ref{thm:recurrence}]
		Let
		\[
		a_m:=\max_{Y\in\binom{X}{m}}\#(\cS|_Y),
		\]
		and choose $Y_m\in\binom{X}{m}$ such that $\#(\cS|_{Y_m})=a_m$. Set $\cS_m:=\cS|_{Y_m}$. For every integer $m>k+1$, fix a root $r\in Y_m$ and consider the graph $G_r(\cS_m)=(V,E)$. Then
		\[
		\#E=\sum_{x\in Y_m-r}\#\partial_x\cS_m
		\ge (m-1)(a_m-a_{m-1}).
		\]
		Combining this inequality with the hypothesis of Theorem~\ref{thm:recurrence} gives
		\[
		ka_m+cm^{\ell}\ge \#E\ge (m-1)(a_m-a_{m-1}),
		\]
		and hence
		\begin{align}
			a_m
			&\le \frac{m-1}{m-1-k}a_{m-1}+\frac{cm^{\ell}}{m-1-k} \label{eq:recurrence-first}\\
			&\le \left(\frac{m-k}{m-1-k}\right)^k a_{m-1}+\frac{cm^{\ell}}{m-1-k}. \label{eq:recurrence-second}
		\end{align}
		where the second inequality follows from Lemma~\ref{lem:log-concavity}. Dividing by $(m-k)^k$ yields
		\begin{equation}\label{eq:normalized-recurrence}
			\frac{a_m}{(m-k)^k}
			\le
			\frac{a_{m-1}}{(m-1-k)^k}
			+
			\frac{cm^{\ell}}{(m-1-k)(m-k)^k}.
		\end{equation}
		
		Let $m_0=\lfloor k\rfloor+1$, so that $m_0>k$. Iterating the preceding inequality for $m=m_0+1,\ldots,n$ gives
		\begin{equation}\label{eq:iterated-recurrence}
			\frac{a_n}{(n-k)^k}
			\le
			\frac{a_{m_0}}{(m_0-k)^k}
			+
			\sum_{i=m_0+1}^{n}
			\frac{ci^{\ell}}{(i-1-k)(i-k)^k}.
		\end{equation}
		The series on the right converges because its summand is $O(i^{\ell-k-1})$ and $\ell<k$. It is therefore bounded independently of $n$. Consequently,
		\[
		a_n=O(n^k),
		\]
		which proves the theorem.
	\end{proof}
	
	\section{Discussion and open problems}
	
	Theorem~\ref{thm:edge-bound} is nearly sharp in the sense that there exist digraphs representing 2-weakly compatible split systems for which $d(v)=5$ for almost every vertex $v\in V$. Let $X=\{1,\ldots,n\}$, and consider the split system
	\begin{align*}
		\cS={}&\{i,\ldots,j\mid 1,\ldots,i-1,j+1,\ldots,n:2\le i\le j\le n\}\\
		&\cup\{1,i,\ldots,j\mid 2,\ldots,i-1,j+1,\ldots,n:2\le i\le j\le n\}.
	\end{align*}
	Then $\cS$ is 2-weakly compatible and $\#\cS=n(n-1)$. Choose 2 as the root. Locally, the graph resembles a two-layer grid. If $2<i<j<n$, then each of the splits
	\[
	i,\ldots,j\mid 1,\ldots,i-1,j+1,\ldots,n
	\]
	and
	\[
	1,i,\ldots,j\mid 2,\ldots,i-1,j+1,\ldots,n
	\]
	is incident with five distinct edges labelled $1,i,j,i-1,j+1$. Since the number of splits with $i=2$, $i=j$, or $j=n$ is linear in $n$, the average degree of $G_2(\cS)$ is $5-o(1)$. Therefore, the coefficient $5/2$ of $\#V$ in Theorem~\ref{thm:edge-bound} cannot be improved.
	
	Nevertheless, our graph method may still yield a stronger bound. Let $M_2(n)$ denote the maximum size of a 2-weakly compatible split system on $n$ taxa. We propose the following conjecture, which would imply $M_2(n)=o(n^{2+\eps})$ for every $\eps>0$. In fact, the recurrence used in the proof of Theorem~\ref{thm:recurrence} would then yield $M_2(n)=O(n^2\log n)$.
	
	\begin{conjecture}\label{conj:quadratic-error}
		There exists an absolute constant $c>0$ such that, for every 2-weakly compatible split system $\cS$ on an $n$-element set and every associated digraph $G_r(\cS)=(V,E)$,
		\[
		\#E\le 2\#V+cn^2.
		\]
	\end{conjecture}
	
	The main innovation of this paper is the translation of a Tur\'an-type problem on split systems into a graph-theoretic problem. Given a split system $\cS$ on $X$ and a fixed root $r\in X$, we represent each split by the part not containing $r$ and draw a directed edge between two such sets when one is obtained from the other by adding a single element. Such an edge corresponds to an $x$-pair for some $x\in X-r$. This representation allows graph-theoretic methods to be applied to Tur\'an-type problems on split systems: rather than counting $x$-pairs for a fixed $x$, it captures the aggregate structure of all such pairs across all $x$ in a single directed graph. Related extremal problems for cross-free families of sets have been studied extensively; see, for example,~\cite{Fleiner2001,Pevzner1994,Suk2008,KupavskiiPachTomon2019}. We expect the graph representation introduced here to be useful in other Tur\'an-type problems for set systems and split systems.

\end{document}